\documentclass[12pt]{article}
\usepackage[T1]{fontenc}
\usepackage[latin9]{inputenc}
\usepackage[letterpaper]{geometry}
\usepackage{color}
\usepackage{amsmath}
\usepackage{amsthm}
\usepackage{amssymb}
\usepackage{graphicx}

\makeatletter
\numberwithin{figure}{section}
\theoremstyle{plain}
\newtheorem{thm}{\protect\theoremname}[section]
  \theoremstyle{plain}
  \newtheorem{prop}[thm]{\protect\propositionname}
  \theoremstyle{plain}
  \newtheorem{lem}[thm]{\protect\lemmaname}
  \theoremstyle{definition}
  \newtheorem{defn}[thm]{\protect\definitionname}
  \theoremstyle{remark}
  \newtheorem{rem}[thm]{\protect\remarkname}

\usepackage{latexsym}\usepackage{mathrsfs}\usepackage{csquotes}\usepackage{enumerate}
\AtBeginDocument{%
\usepackage{pgfplots}

\renewcommand\equationautorefname{\@gobble}
}
\usepackage{color}
\usepackage{verbatim}

\numberwithin{equation}{section}

\definecolor{Red}{rgb}{1,0,0}

\definecolor{Blue}{rgb}{0,0,1}

\title{Concentration of Geodesics in Directed Bernoulli Percolation}

\author{Christian Houdr\'{e}\thanks{School of Mathematics, Georgia Institute of Technology, Atlanta, Georgia 30332-0160. {\tt Email:\,houdre@math.gatech.edu}. Research supported in part by the grants \# 246283 and \# 524678 from the Simons Foundation.} \and Chen Xu \thanks{School of Mathematics, Georgia Institute of Technology, Atlanta, Georgia 30332-0160. {\tt Email:\,cxu60@math.gatech.edu}.}}
\allowdisplaybreaks

\makeatother

  \providecommand{\definitionname}{Definition}
  \providecommand{\lemmaname}{Lemma}
  \providecommand{\propositionname}{Proposition}
  \providecommand{\remarkname}{Remark}
\providecommand{\theoremname}{Theorem}

\begin{document}
\maketitle

\makeatletter \def\blfootnote{\gdef\@thefnmark{}\@footnotetext} \makeatother


\begin{abstract}
For directed Bernoulli last passage percolation with i.i.d.~weights
on vertices over a $n\times n$ grid and for $n$ large enough, the
geodesics are shown to be concentrated in a cylinder, centered on
the main diagonal and of width of order $n^{(2\kappa+2)/(2\kappa+3)}\sqrt{\ln n}$,
where $1\le\kappa<\infty$ is the curvature power of the shape function
at $(1,1)$. The methodology of proof is robust enough to also apply
to directed Bernoulli first passage site percolation, and further
to longest common subsequences in random words.
\end{abstract}

\section{Introduction}

It has been initially conjectured in \cite{krug1991kinetic} that
many percolation systems including undirected/directed, first/last
passage percolation falls into the KPZ universality class. They are
expected to satisfy the scaling relation: $\chi=2\xi-1$, where $\chi$
and $\xi$ are respectively the shape and the transversal fluctuations
exponents. Moreover, $\chi$ can also be viewed as the asymptotic
order of the standard deviation of the first/last passage time, while
geodesics are expected to be confined to a cylinder around the diagonal
of width of asymptotic order $n^{\xi}$. Specifically, on a two dimensional
$n\times n$ grid, it is conjectured that $\chi=1/3$ and $\xi=2/3$.
However, to date, it has only been shown that $\chi\le1/2$ and that
$\xi\le3/4$, under various types of assumptions. The upper bound
$3/4$ is obtained in \cite{newman1995divergence} by showing that
$2\xi\le1+\chi'$, where $\chi'$ is an exponent closely related to
$\chi$ and is itself upper-bounded by $1/2$ . The relation $2\xi=1+\chi$
has also been recently proved under different definitions of $\xi$
and $\chi$ in \cite{chatterjee2011universal} (see also \cite{auffinger2014simplified}).
As for the bounds for the shape fluctuations exponent $\chi$, fewer
results are available. To date, a sublinear order $O(\sqrt{n/\ln n})$,
in the context of first passage percolation (FPP) with various types
of weight distributions has been shown in \cite{benjamini2011first,damron2015sublinear}.
For a list of other definitions and results on these topics, we refer
the interested readers to the recent comprehensive survey \cite{auffinger201550}. 

Transversal fluctuations have also been studied in a related problem,
i.e., the analysis of the longest common subsequences (LCSs) in two
random words of length $n$. As well known, LCSs can be viewed as
directed last passages in a two-dimensional percolation grid with
dependent Bernoulli weights. It is proved in \cite{houdre2009closeness}
that the optimal alignments corresponding to the LCSs also stay, with
high probability, in a sector close to the diagonal. Moreover, when
it comes to the shape fluctuation, i.e., the standard deviation of
$LC_{n}$, the length of the LCSs, the results are more complete:
First, by the Efron-Stein inequality, the shape fluctuations are upper
bounded by $\sqrt{n}$, for arbitrary distributions on any finite
dictionary. Second, a lower bound of order $\sqrt{n}$ has been obtained
under various asymmetry assumptions (\cite[...]{lember2009standard,HoudreLCSVARLB2012}).
More noticeably, a central limit theorem has been proved for $LC_{n}$
in \cite{HoudreCLT2014}.

In the present paper, we mainly study the transversal fluctuations
in directed last passage percolation (DLPP) and briefly extend it
to other settings. Our methodology shows that, with high probability,
geodesics in DLPP are confined to a cylinder, around the main diagonal,
of width of order $n^{(2\kappa+2)/(2\kappa+3)}\sqrt{\ln n}$, where
$1\le\kappa<+\infty$ is the curvature power of the shape function
at $(1,1)$.

The model under study is the classical one: DLPP on a $n\times n$
grid with $(n+1)^{2}$ vertices, each of which is associated with
a Bernoulli random weight $w$, where $\mathbb{P}(w=1)=s=1-\mathbb{P}(w=0)$,
$0<s<1$, and all the weights are independent. The last passage time
$T(n,n)$ is the maximum of the sums of all the weights along all
unit-step up-right paths on the grid, from $(0,0)$ to $(n_{1},n_{2})$.
For convenience, the path is considered $left-open-right-closed$,
i.e., the weight on $(0,0)$ is excluded:
\[
T(n,n)=\max_{\pi\in\Pi}\sum_{v\in\pi\backslash{(0,0)}}w(v),
\]
where $\Pi$ is the set of all unit-step up-right paths from $(0,0)$
to $(n,n)$, and where each unit-step up-right path $\pi\in\Pi$ is
viewed as an ordered set of vertices, i.e., $\pi=\{v_{0}=(0,0),v_{1},...,v_{2n}=(n,n)\}$
such that $v_{i+1}-v_{i}$ ($i\in[2n-1]$) is either $\textbf{e}_{1}:=(1,0)$
or $\textbf{e}_{2}:=(0,1)$, and $w:\ v\rightarrow w(v)\in\{0,1\}$
is the random weight associated with the vertex $v\in[n]\times[n]$,
where $[n]:=\{0,1,2,...,n\}$. Hereafter \textit{directed path} is
short for unit-step up-right path and any directed path realizing
the last passage time is called a \textit{geodesic}. We also use the
notation $T(V_{1},V_{2})$ to denote the directed last passage time
for a rectangular grid from the lower-left vertex $V_{1}$ to the
upper-right vertex $V_{2}$ ($w(V_{1})$ is also excluded) and sometimes
use coordinates to express $V_{1}$ and $V_{2}$, $e.g.$, when $V_{1}=(i,j)$
and $V_{2}=(k,l)$, $T(V_{1},V_{2}):=T((i,j),(k,l))$.

Let us now briefly describe the content of the paper: in the next
section, we present properties of the shape function of DLPP and state
our main result (Theorem \ref{thm:mainmain}). Section \ref{Sec:skewblocks}
first introduces a way of decomposing the entire grid into blocks
in such a way that, with high probability, most of the blocks in any
optimal decomposition are close-to-square shaped. Next, an intermediate
rate of convergence result used in the proof of the main theorem is
further obtained. Finally, we exhibit two lines $\ell_{1}$ and $\ell_{2}$
respectively above and below the main diagonal, bounding a sector
within which, with high probability, geodesics are confined. Then,
by finely tuning the slopes of these two bounding lines, we produce
a concentration inequality for the fluctuations of the geodesics away
from the main diagonal. In the concluding Section \ref{sec:concluding-remarks},
extensions are briefly stated for the geodesics in directed first
passage percolation (DFPP). Then, the case of LCSs is presented and
some potential refinements are also discussed.

\section{Preliminaries and Main Results\label{Sec:Preliminaries}}

In this section, we introduce the shape function $g$ and a modification
$g_{\bot}$ ( $g$-perp) along with some of their properties. It is
well known that, by superadditivity and Fekete's Lemma, the non-negative
limit 
\[
\lim_{n\rightarrow\infty}\frac{\mathbb{E}T(nx,ny)}{n}=\limsup_{n\rightarrow\infty}\frac{\mathbb{E}T(nx,ny)}{n}:=g(x,y)
\]
exist for any $x,y\in\mathbb{R}^{+}$. The function $g$ is typically
called the shape function, and by a further application of superadditivity,
it can be shown to be concave (see \cite{martin2004limiting}). Instead
of studying $g$ directly, we are more interested in its orthogonal
modification, i.e., in the function $g_{\perp}$, given by $g_{\perp}(q)=g(1-q,1+q)$,
where $q\in(-1,+1)$. Since the transformation $(1-q,1+q)$ is linear,
it is trivial to transfer results from $g$ to $g_{\perp}$. Therefore,
from \cite{martin2004limiting}:
\begin{prop}
$g_{\perp}$ is non-negative and concave.
\end{prop}
By the invariance of $g$ under any permutation of its coordinates,
i.e., since $g(x,y)=g(y,x)$, $g_{\perp}$ is symmetric about $q=0$.
Also $g_{\perp}((-1)^{+})=g_{\perp}(1^{-})=g(0,2)=2s$. Still, by
concavity, $g_{\perp}$ is non-decreasing on $(-1,0]$ and non-increasing
on $[0,1)$ and so $g_{\perp}$ attains its maximum at $q=0$. The
uniqueness of this maximum is not guaranteed but would follow from
the strict concavity of the shape function $g$ at $(1,1)$ which
has been conjectured, in particular, for i.i.d.~Bernoulli weights.
To date, strict concavity has not been proved for any weight distribution.
However, in the setting of undirected FPP, a class of weight distributions
has been shown (see \cite{durrett1981shape}) to produce a shape function
having a flat edge around the direction $(1,1)$. This class of weights
is further studied and more properties of the associated shape function
are obtained in \cite{marchand2002strict,zhang2008shape,zhang2010concentration,auffinger2013differentiability}.
Our first result Theorem \ref{thm:eventa} stating that, with probability
exponentially close to one, geodesics are bounded away from the upper-left
and lower-right corners of the grid, does not requires a strict-concavity
assumption. Instead, it merely requires the existence of a threshold
$t>0$ such that if $q\in(-1,-t)\cup(t,1)$, then $g_{\perp}(q)\lneqq g_{\perp}(0)=g(1,1)$,
i.e., that $g_{\perp}$ is not identically constant on $(-1,1)$.
Before tackling this threshold problem, let us better estimate $g_{\perp}(0)=g(1,1)$. 

First, it is clear that any directed path from $(0,0)$ to $(n,n)$
in a $n\times n$ grid covers exactly $2n$ vertices and the expected
passage time associated with up-right path is $2ns$. But, clearly,
the passage time associated with any such up-right path is at most
the last passage time. Thus $\mathbb{E}T(n,n)\ge2ns$. Therefore,
$g_{\perp}(0)\ge2s$, however this lower bound is strict. 
\begin{lem}
\label{lem:.lbg}$g_{\perp}(0)-2s\ge s(1-s)$.
\end{lem}
\begin{proof}
Consider the diagonal blocks in the $n\times n$ table, i.e., the
$n$ blocks of size $1\times1$ on the diagonal as in Figure \ref{fig:nonflat}.
\begin{figure}
\begin{centering}
\includegraphics[scale=0.7]{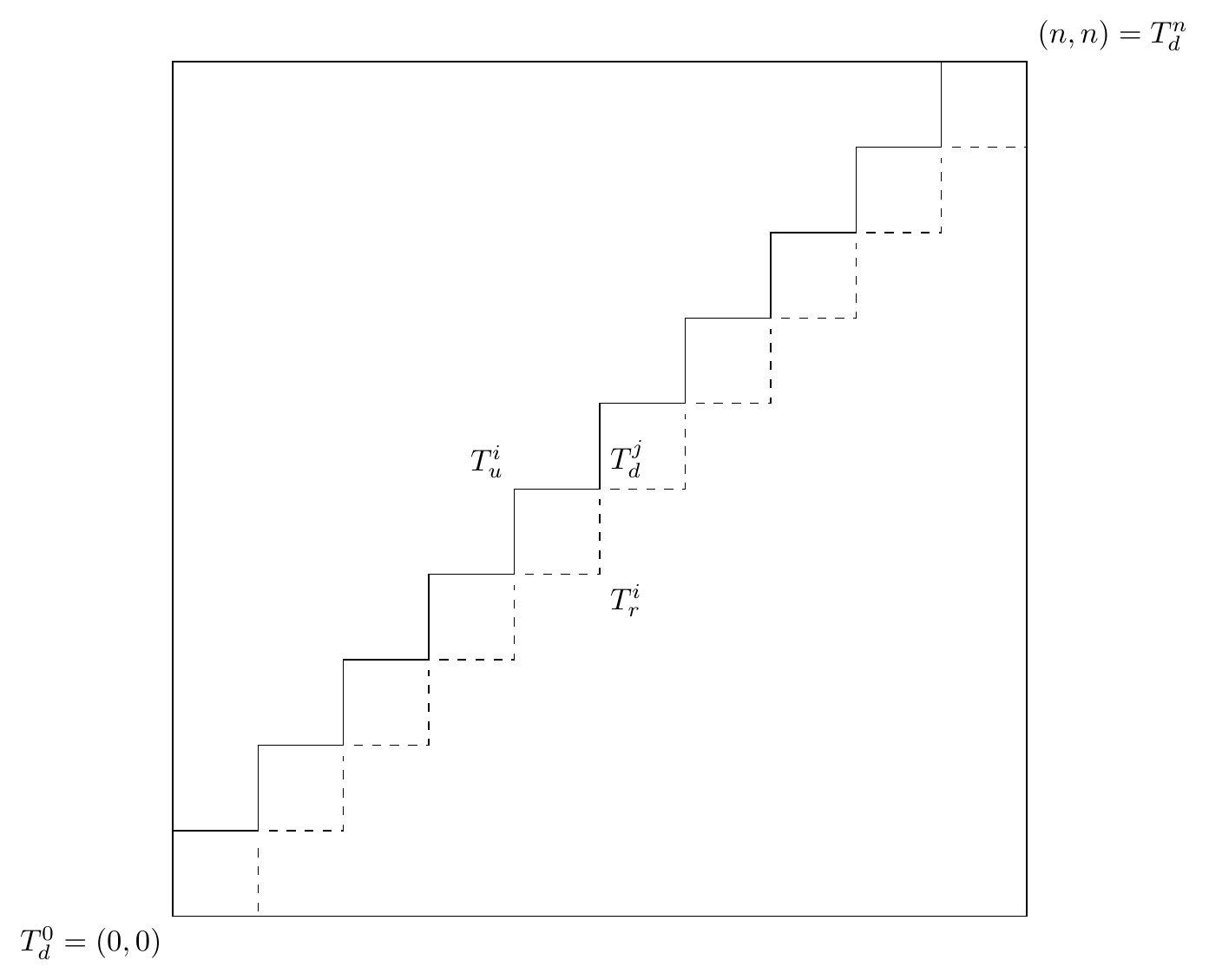}
\par\end{centering}
\caption{}

\label{fig:nonflat}
\end{figure}
Any up-right path on this block goes either up-right or right-up.
Denote by $T_{u}^{i}$ the weight associated with the vertex at the
upper-left corner of the $ith$ $1\times1$ diagonal block, while
$T_{r}^{i}$ is the weight associated with the corresponding lower-right
corner for $i\in[n-1]$ and $T_{d}^{j}$ is the weight associated
with the vertex on the diagonal for $j\in[n]$. Then, all these $3n$
random weights are i.i.d.~Bernoulli random variables with parameter
$s$. Moreover, the maximal passage time of all the paths going inside
these blocks is a lower bound for the last passage time, i.e.,
\[
T(n,n)\ge\sum_{j=1}^{n}T_{d}^{j}+\sum_{i=0}^{n-1}\left(T_{u}^{i}\vee T_{r}^{i}\right).
\]
Hence, 
\begin{align*}
g(1,1) & \ge\lim_{n\rightarrow\infty}\frac{1}{n}\mathbb{E}\left(\sum_{j=1}^{n}T_{d}^{j}+\sum_{i=0}^{n-1}\left(T_{u}^{i}\vee T_{r}^{i}\right)\right)\\
 & \ge\lim_{n\rightarrow\infty}\frac{1}{n}\left(ns+n(1-(1-s)^{2})\right)\\
 & =3s-s^{2}.
\end{align*}
\end{proof}
An explicit expression for $g$ is known for geometric or exponential
weights but not for Bernoulli weights (e.g., see \cite{seppalainen1996hydrodynamic,1998math......1068J,rost1981non}).
So to obtain a specific threshold $t$, as described above, we combine
the lower bound on $g_{\bot}(0)$ obtained in Lemma \ref{lem:.lbg}
with an upper bound on $g$ obtained in \cite{martin2004limiting}.
\begin{prop}
\label{prop:nonflatterninterval}Let $t=1-(g(1,1)-2s)^{2}/8s(1-s)<1-s(1-s)/8$.
Then, for any $q\in(-1,-t)\cup(t,1)$, $g_{\perp}(q)\lneq g_{\perp}(0)=g(1,1)$.
\end{prop}
\begin{proof}
First by Lemma $4.1$ in \cite{martin2004limiting},
\[
g(1,y)\le(1+y)s+2\sqrt{y(1+y)}\sqrt{s(1-s)}.
\]
Without loss of generality, assume $q>0$, thus
\begin{eqnarray}
g_{\perp}(q) & = & (1+q)g\left(\frac{1-q}{1+q},1\right)\nonumber \\
 & = & (1+q)g(1,\frac{1-q}{1+q})\nonumber \\
 & \le & (1+q)\left(\frac{2}{1+q}s+2\sqrt{\frac{2(1-q)}{(1+q)^{2}}}\sqrt{s(1-s)}\right)\nonumber \\
 & = & 2s+2\sqrt{2-2q}\sqrt{s(1-s)}.\label{eq:upbdg}
\end{eqnarray}

When $t=1-(g(1,1)-2s)^{2}/8(1-s)s$, the upper bound on $g_{\perp}(q)$
given in (\ref{eq:upbdg}) is equal to $g_{\perp}(0)=g(1,1)$. Moreover,
by Lemma \ref{lem:.lbg}, 
\[
g(1,1)-2s=g_{\perp}(0)-2s\ge s(1-s),
\]
 which implies that $t\le1-s(1-s)/8$.
\end{proof}
It is commonly believed that, for Bernoulli weights, with sufficiently
small parameter, $g$ is strictly concave. In our setting instead,
and in order to obtain our main result, the finiteness of the curvature
power in the $(1,1)$ direction is imposed. This assumption is stronger
than strict concavity since the curvature power of the shape function
is defined as:
\begin{defn}
The shape function $g$ is said to have curvature power $\kappa(\textbf{e})$
at $\textbf{e\ \ensuremath{\in\mathbb{R}}}^{+}\times\mathbb{R}^{+}$,
if $g$ is differentiable at $\text{\textbf{e}}$ and there exists
$\delta>0$, such that for any $\textbf{z\ \ensuremath{\in\mathbb{R}\times\mathbb{R}}}$
such that $|\textbf{z}|<\delta$ and $\textbf{z}+\textbf{e}/g(\textbf{e})\in L(\textbf{e})$,
where $L(\textbf{e})$ is a supporting line for $g$ at $\textbf{e}$,
\[
c|\textbf{z}|^{\kappa(\textbf{e})}\le|g(\textbf{e}+\textbf{z})-g(\textbf{e})|\le C|\textbf{z}|^{\kappa(\textbf{e})},
\]
for some positive constants $c$ and $C$ depending only on $\delta$.
Otherwise, if $g$ is not differentiable at $\textbf{e}$, set $\kappa(\textbf{e})=1$.
Hereafter, $\kappa\ is\ short\ for\ \kappa((1,1))$. 

By symmetry, the supporting line $L(1,1)$ is in the direction of
$(1,-1)$. Hence, the definition of $\kappa$ is equivalent to the
fact that there exists $\delta>0$, such that for any $\textbf{z\ \ensuremath{\in\mathbb{R}\times\mathbb{R}}}$
satisfying $\textbf{z}\cdot(1,1)=0$ and $|\textbf{z}|<\delta$,
\[
c|\textbf{z}|^{\kappa}\le|g((1,1)+\textbf{z})-g(1,1)|\le C|\textbf{z}|^{\kappa},
\]
for some positive constants $c$ and $C$ depending only on $\delta$.
Then, requiring that $1\le\kappa<+\infty$ is in turn equivalent to:
there exists $\delta>0$ such that for any $q\in(-1,1)$ with $|q|<\delta$,
\begin{equation}
c|q|^{\kappa}\le|g_{\perp}(q)-g_{\perp}(0)|\le C|q|^{\kappa},\label{eq:assumgperp}
\end{equation}
for some positive constants $c$ and $C$, depending only on $\delta$.
\end{defn}
As already indicated, it is believed (e.g., see \cite{kesten1986aspects})
that when $s<s_{c}$, where $s_{c}$ is the critical probability for
the directed last passage percolation with i.i.d.~Bernoulli weights,
the shape function $g$ has curvature power $\kappa=2$. But as mentioned
before Lemma \ref{lem:.lbg}, in FPP, a class of weights has been
shown to be such that $g$ has flat edges, i.e., there exist infinitely
many $\textbf{e}$ such that $\kappa(\textbf{e})=\infty$ (see \cite{durrett1981shape}).
It was first proved in \cite{newman1995divergence} that there exists
$q\in(-1,1)$ such that the lower inequality in (\ref{eq:assumgperp})
holds when $\kappa(1-q,1+q)=2$, while \cite{chatterjee2011universal}
shows that there is a (possibly different) $q$ at which the upper
inequality in  (\ref{eq:assumgperp}) holds when $\kappa(1-q,1+q)=2$.
To finish this section, we state the main result of this paper.
\begin{thm}
\label{thm:mainmain} Let the curvature power $\kappa$ of the shape
function $g$ at $(1,1)$ be such that $1\le\kappa<+\infty$. Then,
in a $n\times n$ grid, with probability exponentially close to $1$,
all the geodesics are within the cylinder, centered on the main diagonal
and of width $O(n^{\frac{2\kappa+2}{2\kappa+3}}\sqrt{\ln n})$.
\end{thm}
As far as notations are concerned, and as usual, $a_{n}=O(b_{n})$
is short for there exists a positive constant $C$ such that $|a_{n}|\le C|b_{n}|$,
for $n$ large enough; $a_{n}=\Theta(b_{n})$ is short for there exist
$0<c<C<+\infty$ such that $cb_{n}\le a_{n}\le Cb_{n}$, for $n$
large enough; $a_{n}=\Omega(b_{n})$ is short for there exists a constant
$K>0$ such that $a_{n}\ge Kb_{n}$, for $n$ large enough and finally,
$a_{n}=o(b_{n})$ is short for $\lim_{n\rightarrow+\infty}\lvert a_{n}\rvert/\lvert b_{n}\rvert=0$.

\section{Proof of Main Results\label{Sec:skewblocks}}

In this section we start by introducing our main tools, i.e., decompositions
and blocks, and then prove some concentration results which are further
related to the concentration of geodesics needed to obtain our main
result.

\subsection{Blocks, Decompositions and Concentration}

Throughout the rest of this manuscript, let $n=mk$ so that the $x$-axis
of the grid is divided into $m$ segments each of equal length $k$.
Meanwhile, the $y$-axis of the grid is also divided into $m$ segments.
The $(m+1)$-tuples $\overrightarrow{r}=(r_{0},r_{1},...,r_{m})$
made of the end points of these consecutive segments on the $y$-edges,
i.e.,
\[
r_{0}=0\le r_{1}\le r_{2}\le...\le r_{m-1}\le r_{m}=n,
\]
 is called a $decomposition$ of the $y$-axis. This decomposition
leads to a decomposition of the grid into $m$ rectangular blocks,
of which the $ith$ ($i=1,2,...,m)$ block has lower-left corner $((i-1)k,r_{i-1})$
and upper-right corner $(ik,r_{i})$, and is of size $k\times(r_{i}-r_{i-1})$.
Moreover, the last passage time associated with the decomposition
$\overrightarrow{r}$ is defined as the summation over all the $m$
last passage times in these $m$ blocks, i.e.,
\[
T_{n}(\overrightarrow{r})=\sum_{i=1}^{m}T(((i-1)k,r_{i-1}),(ik,r_{i})).
\]
By superadditivity, it is clear that $T(n,n)\ge T_{n}(\overrightarrow{r})$.
Moreover, as explained next, there always exists a decomposition $\overrightarrow{r_{*}}$
such that $T(n,p_{0}n)=T_{n}(\overrightarrow{r_{*}})$, and such a
decomposition is called $optimal$. Indeed, one can construct an $optimal$
decomposition $\overrightarrow{r_{*}}$ by taking vertices $(ik,r_{i})$,
$i=0,1,...,m$, on a geodesic for the entire $n\times n$ grid. Heuristically,
any optimal decomposition $\overrightarrow{r_{*}}$ should, roughly,
be evenly distributed over $n$, i.e., all the $m$ blocks in any
optimal decomposition should be mostly square shaped at least with
high probability. To be more precise, let us fix $0<\eta<1$ and $p_{i}>0$
($i=1,2$) such that $0<p_{1}<1<p_{2}$. Let $R_{\eta,p_{1},p_{2}}$
be the deterministic set of decompositions $\overrightarrow{r}$ such
that 
\begin{equation}
\#\{i\in[m]:\ kp_{1}\le r_{i}-r_{i-1}\le kp_{2}\}\ge(1-\eta)m,\label{eq:defeventR}
\end{equation}
in words $R_{\eta,p_{1},p_{2}}$ represents the decompositions having
a proportion of at least $(1-\eta)$ of those $m$ blocks close-to-square
shaped, i.e., the decompositions for which the slope of the block
diagonal is close to $1$, i.e., the non-skewed decompositions. Finally,
let $A_{\eta,p_{1},p_{2}}^{n}$ be the event that all the optimal
decompositions are in $R_{\eta,p_{1},p_{2}}$, i.e., if $\overrightarrow{r}$
is optimal, then 
\begin{equation}
\overrightarrow{r}\in R_{\eta,p_{1},p_{2}}.\label{eq:defeventA}
\end{equation}

Next, we show a lemma asserting that for any decomposition $\overrightarrow{r}=(r_{0},r_{1},...,r_{m})\in R_{\eta,p_{1},p_{2}}^{c}$,
the difference between the expected overall last passage time and
the expected last passage times associated with $\overrightarrow{r}$
is at least linear in $n$. Before proving it, it is shown that even
when the vertex weights, belonging to a particular set to be specified,
are independently resampled, the absolute change in the last passage
time can be upper-bounded by $1$. To specify such a set, declare
the vertices $\{V_{i}=(X_{i},Y_{i})\}_{i=1}^{k}$ to be strictly decreasing,
if there exists a permutation $\pi$ of $\{1,2,...,k\}$ such that
\[
V_{\pi(1)}\prec V_{\pi(2)}\prec...\prec V_{\pi(k)},
\]
where $V_{i}\prec V_{j}$ indicates that both $X_{i}<X_{j}$ and $Y_{i}<Y_{j}$.
For example, on a $n\times n$ grid, the set of all the vertices on
the reversed diagonal, i.e., $\{(n-i,i)\}_{i=0}^{n}$ is a strictly
decreasing set and its cardinality is $n+1$.
\begin{lem}
\label{lem:revdiag} Let a rectangular grid have lower-left vertex
$V_{1}$ and upper-right vertex $V_{2}$ and let $S$ be a strictly
decreasing set of vertices on the grid. Then, the absolute difference
between the last passage times in the original weights setting and
in the modified weights setting, where the weights on $S$ are independently
resampled, is upper bounded by $1$, i.e.,
\[
|T(V_{1},V_{2})-T^{S}(V_{1},V_{2})|\le1,
\]
where $T(V_{1},V_{2})$ and $T^{S}(V_{1},V_{2})$ are respectively
the last passage times before and after resampling.
\end{lem}
\begin{proof}
Let $\Pi$ be the set of all up-right paths from $V_{1}$ to $V_{2}$.
Since $S$ is a set of strictly decreasing vertices, for any path
$\pi\in\Pi$ viewed as a set of vertices, the intersection between
$\pi$ and $S$ is either empty or contains exactly one element, i.e.,
$\#(\pi\cap S)\le1$. Thus,
\[
T_{\pi}^{S}(V_{1},V_{2})-T_{\pi}(V_{1},V_{2})\le1,
\]
where the upper bound is $1$ if and only if there is a vertex $v\in\pi\cap S$
such that $w(v)=0$ and $w^{S}(v)=1$. Let $\pi_{*}^{S}$ be a geodesic
after resampling $S$, i.e., $T_{\pi_{*}^{S}}^{S}(V_{1},V_{2})=T^{S}(V_{1},V_{2})$.
It follows that
\begin{eqnarray*}
T^{S}(V_{1},V_{2}) & = & T_{\pi_{*}^{S}}^{S}(V_{1},V_{2})\\
 & \le & T_{\pi_{*}^{S}}(V_{1},V_{2})+1\\
 & \le & \max_{\pi\in\Pi}T_{\pi}(V_{1},V_{2})+1\\
 & \le & T(V_{1},V_{2})+1.
\end{eqnarray*}
Symmetrically, $T(V_{1},V_{2})-T^{S}(V_{1},V_{2})\le1$ and thus 
\[
|T(V_{1},V_{2})-T^{S}(V_{1},V_{2})|\le1.
\]
\end{proof}
For further convenience, we introduce a transformed shape function
$g_{\leftthreetimes}$ which depends on the slope of the main diagonal
of the grid. Specifically, for $p>0$, set
\[
g_{\leftthreetimes}(p):=\lim_{n\rightarrow\infty}\frac{\mathbb{E}T(n,np)}{n(1+p)/2}.
\]
Now, recalling that $g_{\perp}:\ q\in(-1,1)\rightarrow g_{\perp}(q)\in(0,\infty)$
is defined via
\[
g_{\perp}(q)=g(1-q,1+q)=\lim_{n\rightarrow\infty}\frac{\mathbb{E}T(n-nq,n+nq)}{n},
\]
it is clear that
\[
g_{\leftthreetimes}(p)=g_{\perp}\left(\frac{p-1}{p+1}\right),
\]
for $p\in(0,+\infty)$. To prove a result showing that the difference
of the expectations is at least linear in $n$, in addition to Lemma
\ref{lem:revdiag}, a rate of convergence result for $\mathbb{E}T_{n}/n$
is also needed. This is stated and proved next, with a proof adapted
from \cite{rhee1995rates}.
\begin{prop}
\label{prop:rateofconv}$0\le g_{\leftthreetimes}(1)-\mathbb{E}T_{n}/n\le c\sqrt{\ln n/n}$,
where $c>0$ is an absolute constant.
\end{prop}
\begin{proof}
Consider the last passage time $T_{kn}$ of site percolation on a
$kn\times kn$ grid. A sequence of vertices $\overrightarrow{V}=(V_{1}=(X_{1}=0,Y_{1}=0),\ V_{2},...,\ V_{k}=(X_{k}=kn,Y_{k}=kn))$
is called a $partition$ of the grid, if
\begin{align}
0=X_{1}\le X_{2}\le...\le X_{k} & =kn,\label{eq:partition_def_1}\\
0=Y_{1}\le Y_{2}\le...\le Y_{k} & =kn,\label{eq:partition_def_2}\\
||V_{i}-V_{i+1}||_{1} & =2n,\label{eq:partition_def_3}
\end{align}
where $||\cdot||_{1}$ denotes the $\ell_{1}$-distance. Further,
let the last passage time associated with some partition $\overrightarrow{V}$
be 
\[
T(\overrightarrow{V})=\sum_{i=0}^{k-1}T(V_{i},V_{i+1}).
\]
Then, as proved next,
\begin{equation}
T_{kn}=\max_{partitions\ \overrightarrow{V}}T(\overrightarrow{V}).\label{eq:identity}
\end{equation}
First, it is clear that the identity is true if only (\ref{eq:partition_def_1})
and (\ref{eq:partition_def_2}) are imposed on partitions. To show
it is fine to include (\ref{eq:partition_def_3}), it suffices to
show that any geodesic can be divided into $k$ segments such that
the $\ell_{1}$-distance between two ends of any segment is exactly
$2n$. Assume some geodesic is an ordered set of $2kn+1$ vertices
$(W_{0}=(0,0),W_{2},W_{3},...,W_{2kn}=(kn,kn))$. Notice that $||W_{i}-W_{i+j}||_{1}=j$,
for any $i,\ i+j\in[2kn]$. Therefore, this geodesic can be divided
on $(V_{0}=W_{0},V_{1}=W_{2n},...,V_{k}=W_{2kn})$ into $k$ segments
with $||V_{i}-V_{i+1}||_{1}=2n$. 

Next, consider a particular set of directed paths going from $(0,0)$,
through $(k,2n-k)$, to $(2n,2n)$ on a $2n\times2n$ grid. Then,
by superadditivity, $T(k,2n-k)+T(2n-k,k)\le T(2n,2n)$. Further, thanks
to symmetry, $\mathbb{E}T(k,2n-k)=\mathbb{E}T(2n-k,k)$. Hence, $\mathbb{E}T(k,2n-k)\le\frac{1}{2}\mathbb{E}T(2n,2n).$
So, $\mathbb{E}T(\overrightarrow{V})\le k\mathbb{E}T(2n,2n)/2$. 

On the other hand, let us view $T(\overrightarrow{V})$ as a function
\[
T(\overrightarrow{V}):\ (D_{1},...,D_{2kn})\rightarrow T(\overrightarrow{V})(D_{1},...,D_{2kn})\in\mathbb{N},
\]
where $\{D_{j}\}_{j=1}^{2kn}$ is the set of batches of the weights
$w(v)$ on the same reversed diagonal, i.e., $D_{j}=\{w(v)\ |\ v\in\{x+y=j\}\cap[kn]\times[kn]\}$.
Clearly, the independence of the weights yields the independence of
the random vectors $D_{j}$, $j=1,...,2kn$. Further, any batch $D_{j}$
is a strictly decreasing set of vertices and so by Lemma \ref{lem:revdiag},
independently resampling any one of these random vectors, say, as
$D_{j_{0}}'$ gives 
\[
|T(\overrightarrow{V})(D_{1},...,D_{j_{0}}',...D_{2kn})-T(\overrightarrow{V})(D_{1},...,D_{j_{0}},...,D_{2kn})|\le1.
\]
Further, applying Hoeffding's martingale inequality gives
\[
\mathbb{P}\left(T(\overrightarrow{V})-\mathbb{E}T(\overrightarrow{V})\ge tkn\right)\le\exp\left(-t^{2}kn\right),
\]
and so,
\[
\mathbb{P}\left(T(\overrightarrow{V})-\frac{k}{2}\mathbb{E}T_{2n}\ge tkn\right)\le\exp\left(-t^{2}kn\right).
\]
In addition, by (\ref{eq:identity}), 
\begin{eqnarray}
\mathbb{P}\left(\frac{T_{kn}}{kn}-\frac{\mathbb{E}T_{2n}}{2n}\ge t\right) & = & \mathbb{P}\left(T_{kn}-\frac{k}{2}\mathbb{E}T_{2n}\ge tkn\right)\nonumber \\
 & \le & \sum_{partitions\ \overrightarrow{V}}\mathbb{P}\left(T(\overrightarrow{V})-\frac{k}{2}\mathbb{E}T_{2n}\ge tkn\right)\nonumber \\
 & \le & \#partitions\exp\left(-t^{2}kn\right).\label{eq:rateofconvgence_main}
\end{eqnarray}
Since $\#partitions=\left(\begin{array}{c}
kn+k-1\\
k-1
\end{array}\right)$ and $\left(\begin{array}{c}
p\\
q
\end{array}\right)\le p^{p}/q^{q}(p-q)^{p-q}$, 
\[
\#partitions\le\left(\begin{array}{c}
kn+k\\
k
\end{array}\right)\le(kn+k)^{kn+k}/k^{k}(kn)^{kn}\le\exp(ck\ln n),
\]
for some absolute constant $c>0$. Combining this with (\ref{eq:rateofconvgence_main})
and taking $t=\sqrt{2c\ln n/n}$ leads to
\[
\mathbb{P}\left(\frac{T_{kn}}{kn}-\frac{\mathbb{E}T_{2n}}{2n}\ge\sqrt{\frac{2c\ln n}{n}}\right)\le\exp(-ck\ln n).
\]
Hence,
\[
\frac{\mathbb{E}T_{kn}}{kn}-\frac{\mathbb{E}T_{2n}}{2n}\le\sqrt{\frac{2c\ln n}{n}}+2\exp(-ck\ln n),
\]
and letting $k\rightarrow\infty$ gives,
\[
\frac{\mathbb{E}T_{2n}}{2n}\ge\gamma^{*}-\sqrt{\frac{2c\ln n}{n}}.
\]
In addition, for odd integers, 
\[
\frac{\mathbb{E}T_{2n+1}}{2n+1}\ge\frac{\mathbb{E}T_{2n}}{2n}-\frac{\mathbb{E}T_{2n+1}}{2n(2n+1)}\ge\gamma^{*}-\sqrt{\frac{2c\ln n}{n}}-\frac{1}{n}.
\]
\end{proof}
To state our next lemma, recall that $R_{\eta,p_{1},p_{2}}=\{r:\ \#\{i\in[m]:\ kp_{1}\le r_{i}-r_{i-1}\le kp_{2}\}\ge(1-\eta)m\}.$
\begin{lem}
\label{lem:linske} Let $0<\eta<1$ and let $p_{i}$ ($i=1,2$) be
such that $0<p_{1}<1<p_{2},$ $g_{\leftthreetimes}(p_{i})<g_{\leftthreetimes}(1)$.
Let $\delta^{*}=\min(g_{\leftthreetimes}(1)-g_{\leftthreetimes}(p_{1}),\allowbreak g_{\leftthreetimes}(1)-g_{\leftthreetimes}(p_{2}))$
and let $\delta^{*}\eta=\Omega(\sqrt{\log n/n})$. Then, for any $\overrightarrow{r}=(r_{0},r_{1},...,r_{m})\in R_{\eta,p_{1,}p_{2}}^{c}$
and any $\delta\in(0,\delta^{*})$,
\begin{equation}
\mathbb{E}(T_{n}(\overrightarrow{r})-T_{n})\le-\frac{\delta\eta n}{2},\label{eq:lineardev}
\end{equation}
for all $n=n(\eta,\delta)$ large enough.
\end{lem}
\begin{proof}
Let $p>0$. By superadditivity, $g_{\leftthreetimes}(p)$ is well
defined and finite. Moreover, for any $k\ge1$,
\begin{equation}
\frac{2\mathbb{E}T(k,kp)}{k(1+p)}\le g_{\leftthreetimes}(p).\label{eq:limitslopp}
\end{equation}
Since $g_{\bot}$ is symmetric around $q=0$ and concave and since
$(1-p)/(1+p)$ is a monotone transformation in $p$, $g_{\leftthreetimes}$
is non-decreasing up to $p=1$ and non-increasing thereafter. Proposition
\ref{prop:nonflatterninterval} shows that there exist $0<p_{1}<1<p_{2}$
such that for any $p\notin[p_{1},p_{2}]$,
\begin{equation}
g_{\leftthreetimes}(p)\le\max(g_{\leftthreetimes}(p_{1}),g_{\leftthreetimes}(p_{2})).\label{eq:concave}
\end{equation}
Therefore, for any $p\notin[p_{1},p_{2}]$, (\ref{eq:limitslopp})
and (\ref{eq:concave}) lead to:
\begin{equation}
\frac{2\mathbb{E}T(k,kp)}{k(1+p)}\le\max(g_{\leftthreetimes}(p_{1}),g_{\leftthreetimes}(p_{2}))=g_{\leftthreetimes}(1)-\delta^{*},\label{eq:limitconcave}
\end{equation}
where $0<\delta^{*}:=\min(g_{\leftthreetimes}(1)-g_{\leftthreetimes}(p_{1}),g_{\leftthreetimes}(1)-g_{\leftthreetimes}(p_{2}))$.

From here on, the proof proceeds as the proof, with its notation,
of Lemma $2.1$, in \cite{houdre2009closeness}. Since the weights
are identically distributed, in the $ith$ block $[(i-1)k+1,ik]\times[r_{i-1}+1,r_{i}]$,
letting $r_{i}-r_{i-1}:=kp$ and assuming $(r_{i}-r_{i-1})/k=p\notin[p_{1},p_{2}]$,
then (\ref{eq:limitconcave}) gives
\begin{equation}
g_{\leftthreetimes}(1)-\frac{2\mathbb{E}T(((i-1)k+1,r_{i-1}+1),(ik,r_{i}))}{k+r_{i}-r_{i-1}}\ge\delta^{*}.\label{eq:iblocklimitconcave}
\end{equation}
Hence, 
\[
\frac{1}{2}g_{\leftthreetimes}(1)(k+r_{i}-r_{i-1})-\mathbb{E}T(((i-1)k+1,r_{i-1}+1),(ik,r_{i}))\ge\frac{1}{2}\delta^{*}k.
\]
Letting $\mathcal{M}:=\{i:\ r_{i}-r_{i-1}\notin[kp_{1},kp_{2}]\}$,
we then have
\begin{equation}
\sum_{i\in\mathcal{M}}\frac{1}{2}g_{\leftthreetimes}(1)(k+r_{i}-r_{i-1})-\mathbb{E}T(((i-1)k+1,r_{i-1}+1),(ik,r_{i}))\ge\frac{1}{2}\delta^{*}k\eta m=\frac{1}{2}n\delta^{*}\eta,\label{eq:skewsum}
\end{equation}
while, for any $i\in\mathcal{M}^{c}=\{i:\ r_{i}-r_{i-1}\in[kp_{1},kp_{2}]\}$,
(\ref{eq:limitslopp}) gives
\[
\frac{1}{2}g_{\leftthreetimes}(1)(k+r_{i}-r_{i-1})-\mathbb{E}T(((i-1)k+1,r_{i-1}+1),(ik,r_{i}))\ge0.
\]
Therefore,
\begin{equation}
\sum_{i\in\mathcal{M}}\frac{1}{2}g_{\leftthreetimes}(1)(k+r_{i}-r_{i-1})-\mathbb{E}T(((i-1)k+1,r_{i-1}+1),(ik,r_{i}))\le g_{\leftthreetimes}(1)n-\mathbb{E}T_{n}(\overrightarrow{r}).\label{eq:totalsum}
\end{equation}
Combining (\ref{eq:skewsum}) and (\ref{eq:totalsum}) leads to, 
\begin{equation}
g_{\leftthreetimes}(1)n-\mathbb{E}T_{n}(\overrightarrow{r})\ge\frac{n\delta^{*}\eta}{2},\label{eq:totaldev}
\end{equation}
when $\overrightarrow{r}=(r_{0},r_{1},...,r_{m})\in R_{\eta,p_{1,}p_{2}}^{c}$.
Next, by Proposition \ref{prop:rateofconv} and since $\delta^{*}\eta=\Omega(\sqrt{\ln n/n})$,
it follows that for any positive $\delta<\delta^{*}$ and for all
$n$ large enough,
\begin{equation}
0\le g_{\leftthreetimes}(1)-\frac{\mathbb{E}T_{n}}{n}\le c\sqrt{\frac{\ln n}{n}}\le\frac{(\delta^{*}-\delta)\eta}{2},\label{eq:squarelimt}
\end{equation}
where $c>0$ is an absolute constant. So combining (\ref{eq:totaldev})
and (\ref{eq:squarelimt}), and for $\overrightarrow{r}=(r_{0},r_{1},...,r_{m})\in R_{\eta,p_{1},p_{2}}^{c}$,
\[
\mathbb{E}(T_{n}(\overrightarrow{r})-T_{n})\le-\frac{\delta\eta n}{2}.
\]
\end{proof}

Before presenting the main result of this section, recall (see (\ref{eq:defeventA}))
that $A_{\eta,p_{1},p_{2}}^{n}$ is the event that all the optimal
decompositions belong to $R_{\eta,p_{1},p_{2}}$.
\begin{thm}
\label{thm:eventa}Let $0<\eta<1$ and let $p_{i}$ ($i=1,2$) be
such that $0<p_{1}<1<p_{2},$ $g_{\leftthreetimes}(p_{i})<g_{\leftthreetimes}(1)$.
Let $\delta^{*}=\min(g_{\leftthreetimes}(1)-g_{\leftthreetimes}(p_{1}),\allowbreak g_{\leftthreetimes}(1)-g_{\leftthreetimes}(p_{2}))$
and let $\delta^{*}\eta=\Omega(\sqrt{\log n/n})$. Let the integer
$k$ be such that $(1+\ln k)/k\le\delta^{2}\eta^{2}/16$, where $\delta\in(0,\delta^{*})$.
Then,
\[
\mathbb{P}(A_{\eta,p_{1},p_{2}}^{n})\ge1-\exp\left(-n\left(-\frac{1+\ln k}{k}+\frac{\delta^{2}\eta^{2}}{16}\right)\right),
\]
for all $n=n(\eta,\delta)$ large enough.
\end{thm}
\begin{proof}
The beginning of this proof, which is similar to that of Theorem $2.2$
in \cite{houdre2009closeness}, is only sketched. By superadditivity,
the decomposition $\overrightarrow{r}$ is optimal if and only if
\begin{equation}
T_{n}(\overrightarrow{r})\ge T_{n}.\label{eq:optr}
\end{equation}
Assume now that the event $A_{\eta,p_{1},p_{2}}^{n}$ does not hold.
Then there exists an optimal decomposition $\overrightarrow{r_{*}}$
such that $\overrightarrow{r_{*}}\in R_{\eta,p_{1},p_{2}}^{c}$, i.e.,
\begin{align*}
(A_{\eta,p_{1},p_{2}}^{n})^{c} & =\bigcup_{\overrightarrow{r}\in R_{\eta,p_{1},p_{2}}^{c}}\{\overrightarrow{r}=\overrightarrow{r_{*}}\}\\
 & =\bigcup_{\overrightarrow{r}\in R_{\eta,p_{1},p_{2}}^{c}}\{T_{n}(\overrightarrow{r})-T_{n}\ge0\},
\end{align*}
hence,
\begin{equation}
\mathbb{P}((A_{\eta,p_{1},p_{2}}^{n})^{c})\le\sum_{\overrightarrow{r}\in R_{\eta,p_{1},p_{2}}^{c}}\mathbb{P}(T_{n}(\overrightarrow{r})-T_{n}\ge0).\label{eq:compleupbd}
\end{equation}
Then, by Lemma \ref{lem:linske}, for any decomposition $\overrightarrow{r}\in R_{\eta,p_{1},p_{2}}^{c}$,
\[
\mathbb{E}(T_{n}(\overrightarrow{r})-T_{n})\le-\frac{\delta\eta n}{2},
\]
and so 
\begin{equation}
\mathbb{P}(T_{n}(\overrightarrow{r})-T_{n}\ge0)\le\mathbb{P}\left(T_{n}(\overrightarrow{r})-T_{n}-\mathbb{E}(T_{n}(\overrightarrow{r})-T_{n})\ge\frac{\delta\eta n}{2}\right),\label{eq:hoeffdingpre}
\end{equation}
for all $n$ large enough. Next, as in the proof of Proposition \ref{prop:rateofconv},
we view the random variable $T_{n}(\overrightarrow{r})-T_{n}:=\Delta$
as a function 
\[
\Delta:\ (D_{1},...,D_{2n})\rightarrow\Delta(D_{1},...,D_{2n})\in\mathbb{Z}\cap[-2n,2n],
\]
where $\{D_{j}\}_{j=1}^{2n}$ is the set of batches of the weights
$w(v)$ on the same reversed diagonal, i.e., $D_{j}=\{w(v)\ |\ v\in\{x+y=j\}\cap[n]\times[n]\}$.
So by Lemma \ref{lem:revdiag} again, independently resampling any
one of these random vectors, say, as $D_{j_{0}}'$ gives 

\begin{multline}
|\Delta(D_{1},...,D_{j_{0}}',...D_{2n})-\Delta(D_{1},...,D_{j_{0}},...,D_{2n})|\\
\le|T_{n}^{j_{0}}(\overrightarrow{r})-T_{n}(\overrightarrow{r})|+|T_{n}^{j_{0}}-T_{n}|\le2,\label{eq:martconbd}
\end{multline}
where $T_{n}^{j_{0}}(\overrightarrow{r})$ and $T_{n}^{j_{0}}$ are
respectively the last passage time associated with $\overrightarrow{r}$
and the overall last passage time with the weights in $D_{j_{0}}$
resampled. 

Finally, Hoeffding's martingale inequality applied to $\Delta(D_{1},...,D_{2n})$
and (\ref{eq:hoeffdingpre}) yield
\[
\mathbb{P}(T_{n}(\overrightarrow{r})-T_{n}\ge0)\le\exp\left(-\frac{\delta^{2}\eta^{2}n}{16}\right),
\]
for $\overrightarrow{r}\in R_{\eta,p_{1},p_{2}}^{c}$. Further, by
(\ref{eq:compleupbd}),
\begin{eqnarray*}
\mathbb{P}((A_{\eta,p_{1},p_{2}}^{n})^{c}) & \le & (ek)^{m}\exp\left(-\frac{\delta^{2}\eta^{2}n}{16}\right)\\
 & \le & \mbox{exp}\left(-n\left(-\frac{1+\ln k}{k}+\frac{\delta^{2}\eta^{2}}{16}\right)\right),
\end{eqnarray*}
since 
\[
\#R_{\eta,p_{1},p_{2}}^{c}\le\left(\begin{array}{c}
n\\
m
\end{array}\right)\le\frac{n^{m}}{n!}\le\left(\frac{en}{m}\right)^{m},
\]
when $n$ is large enough.
\end{proof}
\begin{rem}
Note that above, when applying Hoeffding's martingale inequality and
if only a single weight had been independently resampled then, the
exponential concentration would have failed to hold, since this naively
constructed martingale would have had a length of size $\Theta(n^{2})$.
This justifies and motivates resampling weights in batches.

\end{rem}

\subsection{Proof of the Main Result\label{Sec:closeness2diag}}

Heuristically, if most blocks in an optimal decomposition are close-to-square
shaped, then all the vertices on the diagonals of these blocks are
close to the main diagonal of the grid and therefore all the corresponding
geodesics going through these vertices do not deviate much from it.
Further, the parameters such as $k$, $\delta$ and $\eta$ can be
fixed in an optimal way so that the cylinder, in which geodesics are
confined, is as small as possible. 
\begin{proof}[{\textbf{Proof~of~Theorem}} \ref{thm:mainmain}:]
Let $D_{\eta,p_{1},p_{2}}^{n}$ be the event that all the geodesics
are above the line $\ell_{1:\ }y=p_{1}x-p_{1}n\eta-p_{1}k$ and below
the line $\ell_{2}:\ y=p_{2}x+p_{2}n\eta+p_{2}k$. We first show that
the probability of this event is exponentially close to $1$. Again
the proof is similar to the corresponding result in \cite{houdre2009closeness}
and as such only sketched. We start with a few definitions: denote
by $D_{a}^{n}$ the event that all the geodesics are above the line
$\ell_{1}:\ y=p_{1}x-p_{1}n\eta-p_{1}k$ and by $D_{b}^{n}$ the event
that they are below the line $\ell_{2}:\ y=p_{2}x+p_{2}n\eta+p_{2}k$.
Then $D_{\eta,p_{1},p_{2}}^{n}=D_{a}^{n}\cap D_{b}^{n}$, hence
\[
\mathbb{P}((D_{\eta,p_{1},p_{2}}^{n})^{c})\le\mathbb{P}((D_{a}^{n})^{c})+\mathbb{P}((D_{b}^{n})^{c}).
\]
Moreover, as shown next,
\begin{equation}
A_{\eta,p_{1},p_{2}}^{n}\subset D_{a}^{n},\ A_{\eta,p_{1},p_{2}}^{n}\subset D_{b}^{n},\label{eq:adinclusion}
\end{equation}
so that by Theorem \ref{thm:eventa}
\begin{eqnarray}
\mathbb{P}((D_{\eta,p_{1},p_{2}}^{n})^{c}) & \le & 2\exp\left(-n\left(-\frac{1+\ln k}{k}+\frac{\delta^{2}\eta^{2}}{16}\right)\right).\label{eq:eventd}
\end{eqnarray}
To prove (\ref{eq:adinclusion}), at first we prove that $A_{\eta,p_{1},p_{2}}^{n}\subset D_{a}^{n}$.
This last inclusion is obtained by considering three cases which depend
on $x$: If $(x,y)$ is on one of the geodesics in the event $A_{\eta,p_{1},p_{2}}^{n}$,
namely, $x=uk$, where $u\in\mathbb{N}=\{0,1,2,...\}$, and $uk\le n\eta$;
$x=uk$, where $u\in\mathbb{N}$ and $uk>n\eta$; and there exists
$u\in\mathbb{N}$ such that $uk<x<(u+1)k$. Before we move on to verify
the inclusion case by case, recall again that $A_{\eta,p_{1},p_{2}}^{n}$
corresponds to geodesic decompositions belonging $R_{\eta,p_{1},p_{2}}$,
i.e., such that the number of $i\in[m]$ with $(r_{i+1}-r_{i})\in[p_{1}k,p_{2}k]$
is at least $(1-\eta)m$, where $m$ is the total number of blocks
and $mk=n$ (see (\ref{eq:defeventR})).

In the first of these cases, $p_{1}x-p_{1}n\eta\le0$ and therefore,
\[
y\ge0\ge p_{1}x-p_{1}n\eta\ge p_{1}x-p_{1}n\eta-p_{1}k.
\]

In the second case, by the very definition of $R_{\eta,p_{1},p_{2}}$,
there are at most $\eta m$ blocks having side length $(r_{i+1}-r_{i})$
less than $p_{1}k$. Since $x=uk$, in the worst case, all these $\eta m$
blocks appear among the first $u$ blocks. Hence, at least $u-\eta m$
blocks of the first $u$ blocks have side length at least equal to
$p_{1}k$. Therefore,
\[
y\ge(u-\eta m)p_{1}k=p_{1}(uk)-p_{1}\eta mk=p_{1}x-p_{1}\eta n\ge p_{1}x-p_{1}\eta n-p_{1}k.
\]

In the third and last case, since $x_{1}:=uk<x<(u+1)k$, then 
\begin{equation}
x-x_{1}<k.\label{eq:integerx}
\end{equation}
 From the first two cases, $y_{1}\ge p_{1}x_{1}-p_{1}\eta n$. Moreover,
a geodesic is a directed path and so $y\ge y_{1}$ since $x>x_{1}$.
Hence, by (\ref{eq:integerx}) 
\[
y\ge y_{1}\ge p_{1}x_{1}-p_{1}\eta n\ge p_{1}x-p_{1}k-p_{1}n\eta.
\]
 Symmetrically, a reversed inequality can be proved for the upper
bounding line $y=p_{2}x+p_{2}n\eta+p_{2}k$, and then (\ref{eq:eventd})
follows.

Now, let $k=n^{\alpha}$, $p_{1,2}=1\pm n^{-\beta}$, for $0<\alpha,\beta<1$
and so, $\delta^{*}=\min(g_{\leftthreetimes}(1)-g_{\leftthreetimes}(p_{1}),\allowbreak g_{\leftthreetimes}(1)-g_{\leftthreetimes}(p_{2}))=cn^{-\kappa\beta}$,
for some constant $c>0$. Further, set $\delta=\delta^{*}/2=cn^{-\kappa\beta}/2$
and let $\eta=4\sqrt{2}n^{\kappa\beta-\alpha/2}\sqrt{1+\alpha\ln n}/c$
in (\ref{eq:eventd}) be such that $2\kappa\beta<\alpha$. Then, the
condition $\delta^{*}\eta=\Theta(n^{-\alpha/2}\sqrt{\ln n})=\Omega(\sqrt{\ln n/n})$
is satisfied, since $\alpha<1$. Hence,
\[
\mathbb{P}(D_{\alpha,\beta}^{n})\ge1-2\exp(-(1+\alpha\ln n)n^{1-\alpha}),
\]
where $D_{\alpha,\beta}^{n}$ is the event that all the geodesics
are above the line $y=(1-n^{-\beta})(x-cn^{1+\kappa\beta-\alpha/2}\sqrt{1+\alpha\ln n}-n^{\alpha})$
and below the line $y=(1+n^{-\beta})(x+cn^{1+\kappa\beta-\alpha/2}\sqrt{1+\alpha\ln n}+n^{\alpha})$.

Lastly, we will fix the orders of $\alpha$ and $\beta$ so that the
cylinder has minimal width and so that the condition $2\kappa\beta<\alpha$
is satisfied. Notice that the distances at which the lines $\ell_{1,2}$
are from the main diagonal is of the same order as the Euclidean distance
from their intercepts on the left and right edges of the grid to,
respectively, the lower-left vertex $V_{1}$ and the upper-right vertex
$V_{2}$ as pictured in Figure \ref{fig:disvis}. For the lower bounding
line $\ell_{1}$, denoting its intercept on the left edge by $U_{1}^{1}$,
\begin{eqnarray*}
\lvert U_{1}^{1}V_{1}\rvert & = & (1-n^{-\beta})(cn^{1+\kappa\beta-\alpha/2}\sqrt{1+\alpha\ln n}+n^{\alpha})\\
 & = & \Theta(n^{1+\kappa\beta-\alpha/2}\sqrt{\ln n}+n^{\alpha}).
\end{eqnarray*}
Then denoting its intercept on the right edge by $U_{2}^{1}$, whose
$y$-coordinate is $(1-n^{-\beta})(n-cn^{1+\kappa\beta-\alpha/2}\sqrt{1+\alpha\ln n}-n^{\alpha})$,
\begin{eqnarray*}
\lvert U_{2}^{1}V_{2}\rvert & = & (1-n^{-\beta'})n-(1-n^{-\beta})(n-cn^{1+\kappa\beta-\alpha/2}\sqrt{1+\alpha\ln n}-n^{\alpha})\\
 & = & n(n^{-\beta}-n^{-\beta'})+(1-n^{-\beta})(cn^{1+\kappa\beta-\alpha/2}\sqrt{1+\alpha\ln n}+n^{\alpha})\\
 & = & \Theta(n^{1-\beta}+n^{1+\kappa\beta-\alpha/2}\sqrt{\ln n}+n^{\alpha}),
\end{eqnarray*}
since $\beta'>\beta$. Therefore the distance from $\ell_{1}$ to
the diagonal is of order 
\[
n^{(1-\beta)\vee(1+\kappa\beta-\alpha/2)\vee\alpha}\sqrt{\ln n}.
\]
 Symmetrically, a similar result holds true for the upper line $\ell_{2}$.
The minimizing order occurs for $1-\beta=1+\kappa\beta-\alpha/2=\alpha$,
i.e., 
\[
\alpha=(2\kappa+2)\beta=\frac{2\kappa+2}{2\kappa+3}>2\kappa\beta.
\]
Setting $\alpha=(2\kappa+2)/(2\kappa+3)$ and $\beta=1/(2\kappa+3)$
in the event $D_{\alpha,\beta}^{n}$ gives

\[
\mathbb{P}\left(D_{\frac{2\kappa+2}{2\kappa+3},\frac{1}{2\kappa+3}}^{n}\right)\ge1-2\exp\left(-\left(1+\frac{2\kappa+2}{2\kappa+3}\ln n\right)n^{1/(2\kappa+3)}\right),
\]
which completes the proof.
\end{proof}
\begin{figure}
\begin{centering}
\includegraphics{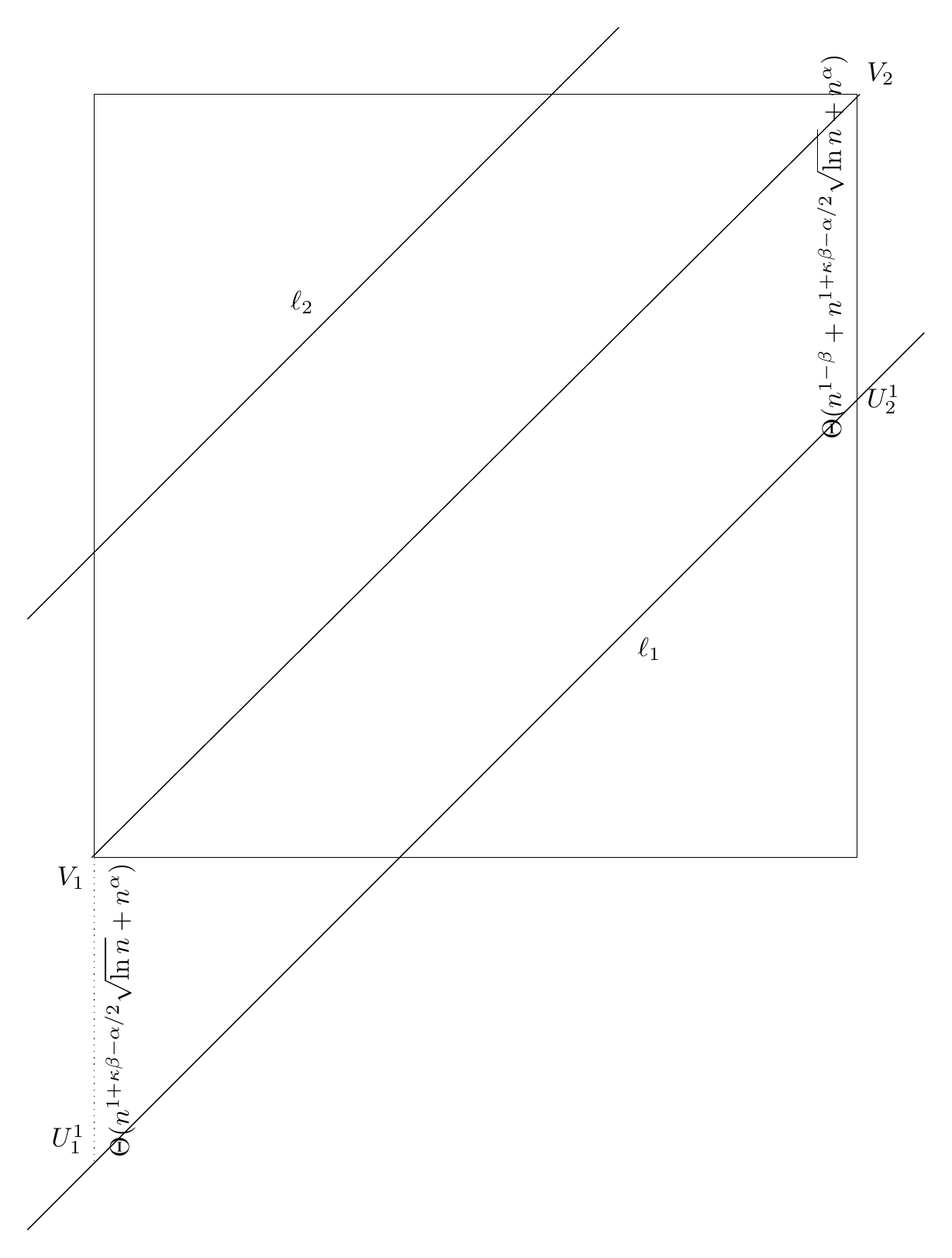}
\par\end{centering}
\caption{}

\label{fig:disvis}
\end{figure}

\section{Concluding Remarks\label{sec:concluding-remarks}}

By symmetry, it is clear that our methodology for proving the concentration
of the geodesics in DLPP is also applicable to the concentration of
geodesics in Bernoulli directed first passage site percolation. In
DFPP, one studies the minimum of the passage times instead of maximum.
In that context, the shape functions $g$, $g_{\perp}$ and $g_{\leftthreetimes}$
are convex instead of concave. Then, a version of Lemma \ref{lem:linske}
with the inequality (\ref{eq:lineardev}) reversed holds true. Further,
a version of Lemma \ref{lem:revdiag} replacing last passage time
by first passage time is still true. Then, so are Theorem \ref{thm:eventa},
Proposition \ref{prop:nonflatterninterval}, Proposition \ref{prop:rateofconv}
and Theorem \ref{thm:mainmain}. Combining all these results finally
leads to:

\begin{thm}
In directed Bernoulli first passage site percolation, let the curvature
power $\kappa$ of the shape function $g$ at $(1,1)$ be such that
$1\le\kappa<+\infty$. Then, in a $n\times n$ grid, with probability
exponentially close to $1$, all the geodesics are within the cylinder,
centered on the main diagonal and of width $O(n^{\frac{2\kappa+2}{2\kappa+3}}\sqrt{\ln n})$.
\end{thm}
To gain a better intuitive view of the concentration order, let, as
commonly believed, $\kappa=2$. Then, the order is $O(n^{6/7}\sqrt{\ln n})$.
Again, it is conjectured that the correct order should be $O(n^{2/3})$
and a currently available bound for the exponent $\text{\ensuremath{\xi}}$
is $3/4$, which has been shown in \cite{newman1995divergence}, in
the setting of first passage percolation on grids in arbitrary dimension.

It is further worth mentioning that our methodology can also be adapted
to produce the order of the closeness to the diagonal for the optimal
alignments corresponding to the LCSs of two random words of size $n$.
In that setting, it is known that the curvature power of the shape
function of the LCSs at $(1,0)$ and $(0,1)$ is equal to $1$ (see
the proof of Lemma $2.1$ in \cite{houdre2009closeness}). However,
the value of $\kappa$ (the curvature power at $(1,1)$) remains unknown
but we conjecture it to be equal to $2$, as in the percolation models.
Adapting our methods leads to:
\begin{thm}
\label{thm:lcsepsilon } In the longest common subsequences problem,
let the curvature power $\kappa$ of the shape function $g$ at $(1,1)$
be such that $1\le\kappa<+\infty$. Then, with probability exponentially
close to $1$, all the alignments corresponding to the longest common
subsequences of two random words of length $n$ are within the cylinder,
centered on the main diagonal and of width of order $O(n^{\frac{2\kappa+2}{2\kappa+3}}\sqrt{\ln n})$.
\end{thm}
Let the exponent of transversal fluctuations $\xi$ be:
\[
\xi=\inf\{\gamma>0:\liminf_{n\rightarrow+\infty}\mathbb{P}(A_{n}^{\gamma})=1\},
\]
where $A_{n}^{\gamma}$ is the event that all the optimal alignments
are confined to a cylinder centered on the main diagonal and of width
of order $n^{\gamma}$. Therefore, from Theorem \ref{thm:lcsepsilon },
for LCSs, $\xi\le(2\kappa+2)/(2\kappa+3)$. Moreover, as previously
mentioned, the shape fluctuations exponent for LCSs has been shown
to be $\chi=1/2$, i.e., $Var(LC_{n})=\Theta(n)$, for various asymmetric
discrete distribution on any finite dictionary (see \cite[...]{HoudreLCSVARLB2012,gong2015lower,lember2009standard}).
But, by the conjectured KPZ universality relation with curvature power
$\kappa$, 

\begin{align*}
\chi & =\kappa\xi-(\kappa-1).
\end{align*}
This leads, for $\chi=1/2$, to $\xi=(2\kappa-1)/(2\kappa)$ which
we conjecture to be equal to $3/4$.

\section*{Acknowledgments}

Many thanks to M. Damron and R. Gong for their bibliographical help
and their numerous comments which greatly helped to improve this manuscript.

\bibliographystyle{plain}
\phantomsection\addcontentsline{toc}{section}{\refname}\bibliography{jointfinal.bbl}

\end{document}